%% file: Amm-Dim-Zer30-06-2014.tex
\documentclass[12pt,reqno]{amsart}

\usepackage{amsmath}
\usepackage{amssymb}
\usepackage{a4wide}
\usepackage{soul}
\usepackage{graphicx}
\usepackage{color} 
\usepackage{graphics}
\usepackage{epsfig}
\usepackage{bbm}

\newtheorem{theorem}{Theorem}[section]
\newtheorem{lemma}[theorem]{Lemma}
\newtheorem{proposition}[theorem]{Proposition}

\newtheorem{remark}[theorem]{Remark}

\newcounter{hypo}

\newenvironment{hyp}{
 \begin{enumerate}
\setcounter{enumi}{\value{hypo}} \item}{\stepcounter{hypo} \end{enumerate}}

\makeatletter
 \@addtoreset{equation}{section}
 \makeatother
 
\title[Rate of decay of the energy]
{Rate of decay of some Petrowsky-like dissipative systems}

\author[K. Ammari]{Ka\"{\i}s Ammari}
\address{Ka\"{\i}s Ammari, UR Analyse et Contr\^ole des Edp, UR 13ES64, D\'ep. de Math\'ematiques, 
Facult\'e des Sciences de Monastir, Universit\'e de Monastir, 
5019 Monastir, Tunisie}
\email{kais.ammari@fsm.rnu.tn}
\author[M. Dimassi]{Mouez Dimassi}
\address{Mouez Dimassi, Universit\'e Bordeaux 1, CNRS, UMR 5251 IMB, 351, cours de la Libération
33405 Talence cedex, France}
\email{dimassi@math.u-bordeaux1.fr}
\author[M. Zerzeri]{Maher Zerzeri}
\address{Maher Zerzeri, Universit\'e Paris XIII, CNRS, UMR 7539 LAGA, 99, ave Jean-Baptiste Cl\'ement
F-93430 Villetaneuse, France}
\email{zerzeri@math.univ-paris13.fr}


\begin{document}

\begin{abstract}
In this paper, we show that the fastest decay rate for some Petrowsky-like dissipative systems is given by the supremum of
the real part of the spectrum of the infinitesimal generator of the underlying semigroup,
if the corresponding operator satisfied some spectral gap condition. We give also some applications to illustrate
our setting.
\end{abstract}

\maketitle

\tableofcontents

\vfill\break

\section{Introduction and Main Result}\label{intro}

The determination of optimal decay rate is
difficult and has not a complete answer in the general case.
In the 1-d case, it was performed mostly,
see \cite{AmHeTu00_01,AmHeTu01_01,BeRa00_01,CaCo01_01,CoZu94_01,CoZu95_01,Fr99_01}, and to references therein.
For higher dimension, G. Lebeau gives in \cite{Le96_01} the explicit
(and optimal) value of the best decay rate in terms of the spectral abscissa of the generator of the
semigroup and the mean value of a damping cofficient along the rays of geometrical optics. 

In this paper, we describe, in abstract setting, the optimal decay rate 
for some Petrowsky-like dissipative systems in terms of spectral quantity
of the corresponding infinitesimal generator of the underlying dynamic.
The main idea is to identify the optimal energy decay rate  with the supremum of the real part of
the associated dissipative operator. 
To do this, it is enough to show that the set of the corresponding generalized eigenvectors
forms a Riesz basis of the energy space. The approach used here is based on some resolvent estimates which 
is obtained by a perturbative method.

In case of the damped wave operator, Cox and Zuazua (\cite{CoZu94_01}) 
adopt the shooting method based  on an ansatz of Horn. 
This  approach consists in  constructing  an explicit  approximation 
of the  characteristic  equation of the underlying system. 
Under the assumption that the damping is of bounded variation, 
they obtained high frequency asymptotic expansions of the spectrum.  
The shooting method can be used only for one-dimensional boundary value problems. 

In the both cases, we require precise knowledge of the spectrum of the corresponding non self-adjoint operators, more precisely,
the behavior of the high frequency set.
The advantage of our approach is that it  works in any dimension  
and in a very general setting (see \cite{Ra81_01} and also \cite{Ka66_01}).

Let us introduce the abstract setting. Let $H$ be a Hilbert space equipped with the norm $\Vert\cdot\Vert_H$. 
Let $A$ be an unbounded operator on $H$, self-adjoint, positive 
and with compact inverse. We denote its domain by $\mathcal{D}(A)$.

Let $B$ be a bounded operator from $U$ to $H$, 
where $\big(U,\Vert \cdot\Vert_{U}\big)$ is another Hilbert space which will be identified with its dual.

We consider the following system:
\begin{equation}\label{Eqd1} 
\left\{
\begin{array}{lc}
\ddot{x}(t) + A x(t) + B B^* \dot{x}(t)= 0,\\
\big(x(0),\dot{x}(0)\big)=(x_0,x_1)\in H_{\frac{1}{2}}\times H,
\end{array}
\right.
\end{equation}
where $t\in[0,\infty)$ is the time and $H_{\frac{1}{2}}=\mathcal{D}(A^{\frac{1}{2}})$ the scaled Hilbert space
with the norm $\Vert z \Vert_{\frac{1}{2}}=\Vert A^{\frac{1}{2}} z\Vert_H$, $\forall z\in H_{\frac{1}{2}}$.
 
From now on, we set ${\mathcal H}:=H_{\frac{1}{2}} \times H$. We endowe this space with the inner product:
$$
\Big\langle \left[f,g\right],\left[u,v\right]\Big\rangle_{\mathcal{H}} := 
\langle A^{\frac{1}{2}}f,A^{\frac{1}{2}}u \rangle_H+ \langle g, v\rangle_H, \quad\text{for all} \,\, [f,g], [u,v]\
\text{in}\ {\mathcal{H}}.
$$ 

We can rewrite the system (\ref{Eqd1}) as a first order differential equation, by putting 
$Y(t)={}^T\big(x(t),\dot{x}(t)\big)$:
\begin{equation}
\label{Eqd3} 
\left\{
\begin{array}{ll}
\dot{Y}(t) + {\mathcal A}_{B} Y(t)=0,\\ 
Y(0)={}^T(x_0,x_1)\in {\mathcal H},
\end{array}
\right.
\end{equation}
where ${\mathcal A}_{B}:={\mathcal A}_0-{\mathcal B}: 
{\mathcal D}({\mathcal A}_{{B}})={\mathcal D}({\mathcal A}_0) \subset {\mathcal H} \rightarrow {\mathcal H},$ with
$$
{\mathcal A}_0= \left(
\begin{array}{cc}
\,\, 0  & I \\
- A \,\,  & 0
\end{array}
\right) : {\mathcal D}({\mathcal A}_0) = {\mathcal D}(A)\times H_{\frac{1}{2}} \subset {\mathcal H}  \rightarrow {\mathcal H},$$ and
$
\displaystyle
{\mathcal B} = \left(
\begin{array}{cc}
0  & \,\, 0 \\
0  & B B^* 
\end{array}
\right)\in {\mathcal L}(\mathcal{H}).$

The operator ${\mathcal A}_0$ is skew-adjoint on ${\mathcal H}$ hence 
it generates a strongly continuous group of unitary operators on ${\mathcal H}$, 
denoted by $\big({\bf S}_0(t)\big)_{t \in \mathbb R}$. Since ${\mathcal A}_{B}$ is dissipative and onto, 
it generates a contraction semi-group on ${\mathcal H}$, denoted by $\big({\bf S}_{B}(t)\big)_{t \in \mathbb R^+}$. 
The system (\ref{Eqd1}) is well-posed. More precisely, the following classical result holds.
\begin{proposition}\label{exist}
Suppose that $(x_0,x_1) \in {\mathcal H}$. Then the
problem (\ref{Eqd1}) admits a unique solution $t\mapsto x(t)$ in the space $C\big([0,+\infty);H_{\frac{1}{2}} \big)\cap C^1\big([0,+\infty);H\big).$
Moreover the solution $t\mapsto x(t)$ satisfies the following energy identity:
\begin{equation}\label{ESTEN}
E\big(x(0)\big)- E\big(x(t)\big)= \int_0^t
\big\Vert B^*\dot{x}(s)\big\Vert_{U}^2\, ds,\quad \hbox{\rm for all} \,\, t\geq 0,
\end{equation}
where $E\big(x(t)\big)= \displaystyle \; \frac{1}{2}\Big\Vert\big(x(t),\dot{x}(t)\big)\Big\Vert^2_{\mathcal H}.$
\end{proposition}
From (\ref{ESTEN}) it follows that the mapping $t\longmapsto \Big\Vert\big(x(t),\dot{x}(t)\big)\Big\Vert^2_{\mathcal H}$ is 
non-increasing. In many applications it is important to know if this
mapping decays exponentially when $t\to+\infty$, i.e., if the
system (\ref{Eqd1}) is exponentially stable. One
of the methods currently used for proving such exponential
stability results is based on an observability inequality for the
conservative system associated to the initial value problem
\begin{equation}\label{Eqd4} 
\ddot\phi(t) + A \phi(t) = 0,\quad 
\big(\phi(0),\dot{\phi}(0)\big)=(x_0,x_1)\in \mathcal H. 
\end{equation}

It is well-known that (\ref{Eqd4}) is well-posed in ${\mathcal H}$. 
The result below, proved in \cite{Ha89_01} (see also \cite{AmTu01_01}), 
shows that the exponential stability of
(\ref{Eqd1}) is equivalent to an observability
inequality for (\ref{Eqd4}). 
\begin{proposition} \label{princx}
The system described by (\ref{Eqd1}) is
exponentially stable in ${\mathcal H}$ if and only if
there exist $T>0,$ and $C_T > 0$ such that
\begin{equation}\label{Esti} 
C_T \int_{0}^{T} \Big\Vert\left(
\begin{array}{lc}
0  & B^* 
\end{array}
\right){\bf S}_0(t) Y_0\Big\Vert^2_{U}\, dt 
\geq  
\big\Vert Y_0\big\Vert^2_{\mathcal H},\quad \hbox{\rm for all}\,\,
Y_0 \in {\mathcal H}.
\end{equation}
\end{proposition}
  
The spectrum of $A$ is given by 
$0<\mu_1\leq \mu_2\leq \mu_3\leq \cdots\leq \mu_n\leq \cdots\rightarrow +\infty$ 
and the family 
$(v_n)_{n\geq 1}$ of corresponding normalized  eigenvectors of $A$ is an orthonormal basis of $H$.  
Now, we can describe the spectrum of the skew-adjoint
operator ${\mathcal A}_{0}$ by the following:
\begin{lemma}\label{undampedspectra}
The eigenvalues of ${\mathcal A}_{0}$ and the corresponding eigenvectors are given by:
\begin{equation}\label{unspectra}
{\mathcal A}_{0}V_{\pm k}=\big(\pm i\sqrt{\mu_k}\big)V_{\pm k},\,\,\,{where}\,\, 
V_{\pm k}=\frac{v_k}{\sqrt{2}} \Big[
\frac{1}{\sqrt{\mu_k}}, \pm i\Big],\quad \hbox{for all}\quad k\in \mathbb N^*.
\end{equation}
Moreover, the family $\big(V_{\pm k}\big)_{k\in \mathbb N^*}$ is an orthonormal basis of the energy
space ${\mathcal H}$.
\end{lemma}

\subsection{Main result}\label{main}

As mentionned in the above, we give the value of the fastest decay rate of 
solution of the equation \eqref{Eqd1},
in terms of the spectral abscissa of the generator ${\mathcal A}_B$. 

Let $\mu(B)$ be the {\it spectral abscissa} of ${\mathcal A}_B$ given by:
\begin{equation}\label{abscissespec}
\mu(\mathcal{A}_B)= \sup \big\{{\rm Re}(\lambda);\ \lambda \in \sigma({\mathcal A}_B) \big\}.
\end{equation}
Here $\sigma({\mathcal A}_B)$
denotes the spectrum of ${\mathcal A}_B$. 
In order to state the result on the optimal decay rate, 
we define the decay rate, depending on $B$, as 
$$
\omega({B})=\inf\big\{\omega;\ \hbox{there exists}\,\, C=C(\omega)>0\, \hbox{such that}$$
\begin{equation}\label{DefRate}
E(x(t))\leq C(\omega) \, e^{2\omega t}E(x(0))
\hbox{ for every solution of (\ref{Eqd1}) with initial data in}\ {\mathcal H}\big\}.
\end{equation}
According to (\ref{ESTEN}), $\omega(B)\leq 0$  (see
\cite{Ha89_01} and also \cite{AmTu01_01}). It follows easily that, 
\begin{equation}\label{IN1} 
\mu(\mathcal{A}_B)\leq \omega(B).
\end{equation}

The following assumption concern the high frequencies of ${\mathcal A}_0$.
More precisely, it deals with the behavior of the gap between two consecutive high frequencies
of ${\mathcal A}_0$. 
For $k\in \mathbb N^*$, we define $\delta_{\pm k}:=\vert \pm i(\sqrt{\mu_{k+1}}-\sqrt{\mu_k})\vert=\sqrt{\mu_{k+1}}-\sqrt{\mu_k}.$ 
We assume that 

\begin{hyp}\label{A1}
$\displaystyle\lim_{k\rightarrow +\infty}\delta_k=+\infty$, 
\end{hyp}

\noindent
and
\begin{hyp}\label{A2}

\,\,\, $\displaystyle \left(\frac{\delta_{k+1}}{\delta_{k}^2}\right)_{k\geq 1}\in l^2(\mathbb N^*)$,
where $ l^2(\mathbb N^*)$ is the space of square integrable sequences.
\end{hyp}

\noindent
{\bf Remarks}

\begin{itemize}
\item[(i)] The assumption \ref{A1} implies that the high frequencies of ${\mathcal{A}_0}$ are simple.

\medskip
\item[(ii)] Assumption \ref{A2} implies
\begin{equation}\label{A3}
\lim_{k\rightarrow +\infty}\left(\frac{\delta_{k+1}}{\delta_{k}^2}\right)=0.
\end{equation}

\medskip
\item[(iii)] Note that, in general, assumption \ref{A2} does not imply hypothesis \ref{A1}.

\end{itemize}

Now, our main result on the optimal decay rate is:
\begin{theorem}\label{Princ}
Assume \ref{A1} and \ref{A2}. Then,
\begin{equation}\label{princ}
\omega(B)=\mu(\mathcal{A}_B).
\end{equation}
\end{theorem}
In other words if all finite energy solutions of (\ref{Eqd1}) are 
exponentially stable then the fastest decay rate 
of the solution of  (\ref{Eqd1}) satisfies (\ref{princ}). 

\noindent
{\bf Outline of the proof:}
In the following, we give an idea of the proof of the main result.

First, for $B=0$, the operator $\mathcal{A}_0$ is skew-adjoint with compact resolvent in 
$\mathcal{H}$. From general operator theory, all
its eigenvalues lie on the imaginary axis and the geometric and algebraic multiplicity
of each eigenvalue are the same. Moreover, there is a sequence
of eigenvectors of $\mathcal{A}_0$  which forms a Riesz (orthonormal, actually) basis
for $\mathcal{H}$.

In our setting, i.e., $B\in {\mathcal L}(U,H)$, we give in Proposition \ref{propspectrales} rough preliminary bounds on the spectrum of $\mathcal{A}_B$.   
Moreover, since $\mathcal{A}_B$ is a bounded perturbation of skew-adjoint operator with compact resolvent it follows from  
\cite[Chapter 5, Theorem 10.1]{GoKr69_01} that the generalized eigenvectors  of $\mathcal{A}_B$ are complete in $\mathcal{H}$.  
For instance, these results are not enough to prove Theorem \ref{Princ}. We need to study the high frequency of $\mathcal{A}_B$, and in particular their algebraic multiplicities.
Using the fact that the distance between two consecutive eigenvalues tends to infinity at infinity, as well as the fact that the dissipation is bounded, we construct in Subsection 
\ref{description} a closed curves $(\Gamma^{(k)})_{\vert k\vert>N_0}$  (for some integer $N_0$ sufficiently large) in the complex plane such that:

\begin{itemize}

\item[(i)] For all $n\in \mathbb N^*$,  $\Gamma^{(\pm n)}$ is centered in ($\pm i\sqrt{\mu_n}$). 
\par\medskip
\item[(ii)] Inside each $\Gamma^{(n)}$ there exits exactly one simple eigenvalue of $\mathcal{A}_B$.
\par\medskip
\item[(iii)] The operator $\mathcal{A}_B$ has exactly $2N_0$ eigenvalues including multiplicity in ${\mathbb C}\setminus(\underset{\vert k\vert>N_0}\cup\Gamma^{(k)})$.

\item[(iv)] $\displaystyle \sum_{\vert k\vert>N_0} \Vert P_{\Gamma^{(k)}}^B-P_{\Gamma^{(k)}}^0\Vert^2_{{\mathcal L}(\mathcal{H})}<\infty$, 
where  $P_{\Gamma^{(k)}}^B$ (resp. $P_{\Gamma^{(k)}}^0$) denotes the Riesz projection associated to $\mathcal{A}_B$ (resp. $\mathcal{A}_0$) 
corresponding to $\Gamma^{(k)}$.

\end{itemize}

The proof of the above  statements  are based on some resolvent estimates of the operators $\mathcal{A}_B$ and $\mathcal{A}_0$. Since the 
generalized eigenvectors  of $\mathcal{A}_B$  are complete and the systems of (generalized)-eigenvectors of $\mathcal{A}_B$ and $\mathcal{A}_0$ are quadratically close in $\mathcal{H}$ (see (iv) above),
it follows  from \cite[Appendix D, Theorem 3]{PoTr86_01} that  the system of generalized eigenvectors  of $\mathcal{A}_B$ constitutes a Riesz  basis in $\mathcal{H}$. 
Now, by a standard argument, we identify the optimal energy decay rate  with the supremum of the real part of $\mathcal{A}_B$, which complete the proof of Theorem \ref{Princ}.

\section{Proof of the main result}\label{proofmain}

As indicated in the introduction,  we will establish Theorem \ref{Princ} by proving that the system of generalized eigenvectors 
of the operator ${\mathcal A}_B$ constitutes a Riesz basis in the energy space ${\mathcal{H}}$,  and that all eigenvalues of ${\mathcal A}_B$
with sufficiently large modulus are algebraically simple. 

\subsection{Description of the spectrum of ${\mathcal{A}_B}$}\label{description}

The operator ${\mathcal A}_{B}$ is a bounded perturbation of a skew-adjoint
operator ${\mathcal A}_{0}$ then, according to \cite[Chapter 5, Theorem
10.1]{GoKr69_01}, we have the following spectral result:
\begin{proposition}\label{propspectrales}
The following properties hold:
\begin{enumerate}
\item
The resolvent of ${\mathcal A}_{B}$ is compact. In particular the spectrum of ${\mathcal A}_{B}$ is discrete, i.e., 
${\mathcal A}_{B}$ has a discrete eigenvalues of finite algebraic multiplicity.
\item
The spectrum of ${\mathcal A}_{B}$ is symmetric about the real axis and is contained in ${\mathcal C}\cup {\mathcal I}$, 
where
\begin{equation}\label{compzonespectre}
{\mathcal C}=\Big\{\lambda\in \mathbb C; \ |\lambda|\geq \sqrt{\mu_1}\,, \,-\beta\leq {\rm Re}(\lambda)\leq 0\Big\}
\end{equation}
\begin{equation}\label{realzonespectre}
{\mathcal I}=\Big[-\beta-\big(\beta^2-\mu_1\big)_+^{\frac12}, \big(\beta^2-\mu_1\big)_+^{\frac12}\Big].
\end{equation}
Here $\beta:=\frac{1}{2}\Vert B^*\Vert^2_{{\mathcal L}(H,U)}<+\infty$, $\mu_1>0,$ 
is the first eigenvalue of $A$ and $(\gamma)_+=\max(\gamma,0)$. 
\item
The root vectors of ${\mathcal A}_{B}$ are complete in ${\mathcal H}$.
\end{enumerate}
\end{proposition}

\begin{proof}

We give only the proof of the second point of the proposition.
Let $\lambda_k:=\lambda_k({\mathcal A}_{B})$ be an eigenvalue of ${\mathcal A}_{B}$. 
We denote by $W(\cdot;\lambda_k)$ the corresponding eigenvector. Then 
$W(\cdot;\lambda_k)=u(\cdot,\lambda_k)^{T}(1,\lambda_k),$ where
$u(\cdot;\lambda_k)$ satisfies 
\begin{equation}\label{eigenvector}
\lambda_k^2u(\cdot;\lambda_k)+\lambda_k BB^*u(\cdot;\lambda_k)+Au(\cdot;\lambda_k)=0
\quad {\rm with}\,\, u(\cdot,\lambda_k)\in H_{1}\,. 
\end{equation}
Since ${\mathcal A}_{B}$ is real it follows that
$\overline{W(\cdot;\lambda_k)}=W\Big(\cdot;\overline{\lambda_k}\Big)$ is an
eigenvector of ${\mathcal A}_{B}$ corresponding to the eigenvalue
$\overline{\lambda_k}$. 
We take the scalar product of the equation (\ref{eigenvector}) with $u(\cdot,\lambda_k)$, we obtain:
$$
\lambda_{\pm k}=
-\frac{1}{2}\Big\Vert B^*\left(\frac{u(\cdot;\lambda_k)}{\Vert u(\cdot;\lambda_k)\Vert_H}\right)\Big\Vert^2_{U}
\pm \left(\frac{1}{4}\Big\Vert B^*\left(\frac{u(\cdot;\lambda_k)}{\Vert u(\cdot;\lambda_k)\Vert_H}\right)\Big\Vert^4_{U}-
\Big\Vert A^{\frac12}\left(\frac{u(\cdot;\lambda_k)}{\Vert u(\cdot;\lambda_k)\Vert_H}\right)\Big\Vert^2_{H}
\right)^{\frac12}.
$$
Hence, if $\lambda_k$ is a non-real eigenvalue, we find
$$
\lambda_{\pm k}=
-\frac{1}{2}\Big\Vert B^*\left(\frac{u(\cdot;\lambda_k)}{\Vert u(\cdot;\lambda_k)\Vert_H}\right)\Big\Vert^2_{U}
\pm
i\sqrt{\Big\Vert A^{\frac12}\left(\frac{u(\cdot;\lambda_k)}{\Vert u(\cdot;\lambda_k)\Vert_H}\right)\Big\Vert^2_{H}-
\frac{1}{4}\Big\Vert B^*\left(\frac{u(\cdot;\lambda_k)}{\Vert u(\cdot;\lambda_k)\Vert_H}\right)\Big\Vert^4_{U}
}\,,
$$
which implies that, since $B^*$ is bounded from $H$ to $U$,
$$
0<-\beta \leq {\rm Re}(\lambda_{\pm k})=-\frac{1}{2}\Big\Vert B^*\left(\frac{u(\cdot;\lambda_k)}{\Vert u(\cdot;\lambda_k)\Vert_H}\right)\Big\Vert^2_{U}\leq 0,
$$
where $\beta:=\frac{1}{2}\Vert B^*\Vert^2_{{\mathcal L}(H,U)}<+\infty$, and
$$
\vert \lambda_{\pm k}\vert^2=\Big\Vert A^{\frac12}\left(\frac{u(\cdot;\lambda_k)}{\Vert u(\cdot;\lambda_k)\Vert_H}\right)\Big\Vert^2_{H}\geq \mu_1. 
$$
If $\lambda_k$ is real we observe that
$$
\sqrt{\frac{1}{4}\Big\Vert B^*\left(\frac{u(\cdot;\lambda_k)}{\Vert u(\cdot;\lambda_k)\Vert_H}\right)\Big\Vert^4_{U}-
\Big\Vert A^{\frac12}\left(\frac{u(\cdot;\lambda_k)}{\Vert u(\cdot;\lambda_k)\Vert_H}\right)\Big\Vert^2_{H}}\leq
\big(\beta^2-\mu_1\big)_+^{\frac12}.
$$
Here $\mu_1>0,$ is the first eigenvalue of $A$.

\end{proof}

To state the principal result of this subsection (see Theorem \ref{countinglemma}), we need to introduce some notations.
For $n\in \mathbb N^*$, we define the three complex numbers: 
\begin{eqnarray}\label{fourpoints}
a_n=\sqrt{\mu_{n-1}}+\frac{1}{2}\delta_{n-1},\quad
\ b_n=\frac{1}{2}\delta_n+i\sqrt{\mu_{n}}
\quad\text{and}\quad 
d_n=-\frac{1}{2}\delta_n+i\sqrt{\mu_{n}},
\end{eqnarray}
where $\delta_k:=\sqrt{\mu_{k+1}}-\sqrt{\mu_k}$ for $k\in \mathbb N^*$.
Let ${\rm Int}(\Gamma^{(n)})$ denote the rectangle with sides  
$\gamma_{1}^{(n)},  \gamma_{2}^{(n)},  \gamma_{3}^{(n)}$ and $ \gamma_{4}^{(n)}$, 
(see Figure 1), where
$$
 \gamma_{1}^{(n)}:=\big\{\lambda\in \mathbb C;\ \text{\rm Im}(\lambda)=a_n \quad \text{and} \quad 
|\text{\rm Re}(\lambda)|<\frac{\delta_n}{2}\big\},
$$
$$
\gamma_{2}^{(n)}:=\big\{\lambda\in \mathbb C;\ \text{\rm Re}(\lambda)=\frac{\delta_n}{2} \quad
\text{and}    \quad   a_n\leq  \text{\rm Im}(\lambda) \leq   a_{n+1}\big\},
$$
$$
\gamma_{3}^{(n)}:=\big\{\lambda\in \mathbb C;\ \text{\rm Im}(\lambda)=a_{n+1}\quad \text{and}\quad 
\text{\rm Re}(\lambda) \ \text{goes from} \ \frac{\delta_n}{2} \ \text{to} \ -\frac{\delta_n}{2}\big\},
$$
and
$$
\gamma_{4}^{(n)}:=\big\{\lambda\in \mathbb C;   \text{\rm Re} (\lambda)=-\frac{\delta_n}{2}\quad
\text{and}\quad  \text{\rm Im}(\lambda) \ \text{goes from} \  a_{n+1} \ \text{to} \  
a_n \big\}.
$$
For  $n=1,2,....$, we set 
\begin{equation}\label{rect}
\Gamma^{(n)}=\gamma_{1}^{(n)}\cup   \gamma_{2}^{(n)}\cup  \gamma_{3}^{(n)}\cup \gamma_{4}^{(n)},\quad
\Gamma^{(-n)}:=\{z\in \mathbb{C};\ \overline{z}\in \Gamma^{(n)}\}
\end{equation}
and
$$
C^{(n)}=\big\{z\in \mathbb C;\ |{\rm Im}(z)|<\sqrt{\mu_{n-1}}+\frac{\delta_{n-1}}{2}\
\text{and} \ |{\rm Re}(z)|<\frac{\delta_{n-1}}{2}\big\}.
$$
Note that by construction ${\rm Int}(\Gamma^{(k)})\cap {\rm Int}(\Gamma^{(n)})=\emptyset$ for all $k,n\in \mathbb Z^*$ such that $k\not=n$.
Here we denote the interior of $\Gamma^{(k)}$ by ${\rm Int}(\Gamma^{(k)})$. Moreover, for all $N\in\mathbb N^*$ we have
${\mathcal C}\cup {\mathcal I}\subset C^{(N)}\bigcup(\underset{{|k|\geq N} }{\cup}{\rm Int}(\Gamma^{(k)}))$, where 
${\mathcal C}$ and ${\mathcal I}$ are given by \eqref{compzonespectre} and \eqref{realzonespectre}.
\begin{figure}
\begin{center}
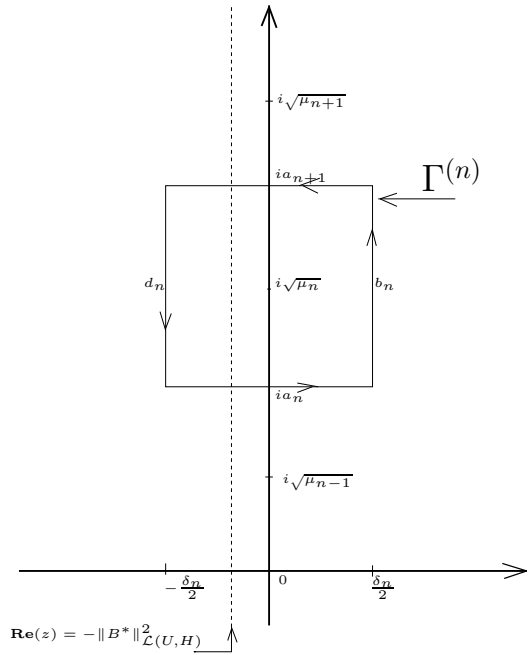
\end{center}
\caption{Location of $\sigma({\mathcal A}_B$)} \label{f1}
\end{figure}

\begin{theorem}\label{countinglemma}
We assume \ref{A1} and that \eqref{A3} is satisfied. Then, there exists $N_0\in \mathbb N^*$ large enough such that
the operator ${\mathcal A}_B$ has exactly $2N_0$ eigenvalues, including multiplicity, in $C_{N_0}$ and 
one simple eigenvalue in ${\rm Int}(\Gamma^{(k)})$ for each $k$ with $|k|>N_0$. This exhausts the spectrum of ${\mathcal A}_B$.
\end{theorem}

We have divided the proof into a sequence of lemmas.

\begin{lemma}\label{resolventestimate}
Assume \ref{A1}. Then, there exists $C>0$ and $N_0\in {\mathbb N}$ (large enough)  such that for $n>N_0$,   
the following properties hold:

\noindent
(i)  $\Gamma^{(\pm n)}\cup\partial C^{(n)}\subset \mathbb C\setminus\big(\sigma({\mathcal A}_B)\cup\sigma({\mathcal A}_0)\big)$.

\noindent
(ii) 
\begin{equation}\label{R0}
\Vert (\lambda-{\mathcal A}_B)^{-1}-(\lambda-{\mathcal A}_0)^{-1}\Vert_{{\mathcal L}({\mathcal H})}\leq
\frac{C}{\delta_{n-1}^2},\, \text { uniformly on } \lambda\in \Gamma^{(\pm n)}\cup  \partial C^{(n)}\,,
\end{equation}
where $\partial C^{(n)}$ is the boundary of the rectangle $C^{(n)}$.
\end{lemma}

\begin{proof}
Since ${\mathcal A}_0$ is skew-adjoint, it follows that
 \begin{equation}\label{resolvent1}
\Vert (\lambda-{\mathcal A}_0)^{-1} \Vert_{{\mathcal L}(\mathcal{H})}\leq \frac{1}{ {\rm dist}(\lambda,\sigma({\mathcal A}_0))}.
\end{equation}
By construction of $\Gamma^{(\pm n)}$ and $C^{(n)}$, we have:
\begin{align}
{\rm dist}  (\Gamma^{(\pm n)},\sigma({\mathcal A}_0))&=
\min\big(|b_n-i\sqrt{\mu_n}|,|d_n-i\sqrt{\mu_n}|,|ia_{n+1}-i\sqrt{\mu_{n+1}}|,|ia_n-i\sqrt{\mu_{n-1}}|\big)\nonumber\\
&=\frac{\delta_{n-1}}{2},\nonumber
\end{align}
and ${\rm dist}  (\partial C^{(n)},\sigma({\mathcal A}_0))\geq  \frac{\delta_{n-1}}{2},$
which  together with (\ref{resolvent1}) yields
 \begin{equation}\label{R1}
\Vert (\lambda-{\mathcal A}_0)^{-1} \Vert_{{\mathcal L}({\mathcal{H}})}\leq \frac{2}{\delta_{n-1}},\, \text { uniformly on } \lambda\in \Gamma^{(\pm n)}\cup \partial C^{(n)}.
\end{equation}
Recalling that ${\mathcal B}={\mathcal A}_0-{\mathcal A}_B$ is a bounded linear operator on ${\mathcal{H}}$ defined by 
$$
{\mathcal B}=\left(\begin{array}{cc}
0 & 0\\
0 & BB^*
\end{array}\right).
$$  
From  \eqref{R1}, we have 
 \begin{equation}\label{R3}
 \Vert  {\mathcal B}(\lambda-{\mathcal A}_0)^{-1}\Vert_{{\mathcal L}({\mathcal H})}\leq
\frac{2\Vert B^*\Vert^2_{\mathcal{L}(H,U)}}{\delta_{n-1}}
\, \text { uniformly on } \lambda\in \Gamma^{(\pm n)}\cup  \partial C^{(n)}\,.
\end{equation}

By \ref{A1}, we choose $N_0$  such that for $n\geq N_0$:
$$
\frac{2\Vert B^*\Vert^2_{\mathcal{L}(H,U)}}{\delta_{n-1}}\leq \kappa<1\,.
$$

Now the first statement  of the lemma follows from \eqref{R1}, \eqref{R3} and the following obvious equality:
 \begin{equation}\label{R4}
\lambda-{\mathcal A}_B=\Big[\text{Id}+{\mathcal B}(\lambda-{\mathcal A}_0)^{-1}\Big](\lambda-{\mathcal A}_0).
\end{equation}
On the other hand \eqref{R4}  yields
$$
(\lambda-{\mathcal A}_B)^{-1}=(\lambda-{\mathcal A}_0)^{-1}+
(\lambda-{\mathcal A}_0)^{-1}\sum_{p\geq 1}\big[-{\mathcal B}(\lambda-{\mathcal A}_0)^{-1}\big]^p,
$$
which together with  \eqref{R1} and  \eqref{R3} imply   \eqref{R0}.
\end{proof}

According to Lemma \ref{resolventestimate}, for $n\geq N_0$  the following Riesz projections are well defined:
\begin{equation}\label{projector}
P_{\Gamma^{(\pm n)}}^{B}:=\frac{1}{2\pi i}\int_{\Gamma^{(\pm n)}}(\lambda-{\mathcal A}_B)^{-1}\, d\lambda, \quad
P_{\Gamma^{(\pm n)}}^{0}:=\frac{1}{2\pi i}\int_{\Gamma^{(\pm n)}}(\lambda-{\mathcal A}_0)^{-1}\, d\lambda,
\end{equation}
 $$
P_{\partial C^{(n)}}^{B}:=\frac{1}{2\pi i}\int_{\partial C^{(n)}}(\lambda-{\mathcal A}_B)^{-1}\, d\lambda\quad \text{and}\quad
P_{\partial C^{(n)}}^{0}:=\frac{1}{2\pi i}\int_{\partial C^{(n)}}(\lambda-{\mathcal A}_0)^{-1}\, d\lambda.
$$
The following result is a simple consequence  of  \ref{A1}, \eqref{A3}, \eqref{R0} and the estimate of the measure on $\partial C^{(n)}$,
$\Gamma^{(\pm n)}$.
\medskip
\begin{lemma}\label{lemR}
We assume \ref{A1} and that \eqref{A3} is satisfied.
Then, there exists $C>0$ (independent of $n$) and $N_0\in {\mathbb N}$ such that for $n\geq N_0$, we have
\begin{equation}\label{estimationprojectors1}
\Vert P_{\Gamma^{(\pm n)}}^{B}-P_{\Gamma^{(\pm n)}}^{0} \Vert_{{\mathcal L}(\mathcal{H})}\leq C\frac{\delta_{n}}{\delta_{n-1}^2}<1,
\end{equation}

\begin{equation}\label{estimationprojectors2}
\Vert P_{\partial C^{(n)}}^{B}-P_{\partial C^{(n)}}^{0} \Vert_{{\mathcal L}(\mathcal{H})}\leq C\frac{\delta_{n}}{\delta_{n-1}^2}<1. 
\end{equation}
\end{lemma}

\noindent
{\it End of the proof of Theorem \ref{countinglemma}}: 
First, recalling that if $P$ and $Q$ are two  projectors  with $\Vert P-Q\Vert<1$, 
then  ${\rm rank}(P)={\rm rank}(Q)$ (see Lemma 3.1 in \cite{GoKr69_01}). 
Thus, in the notation of Lemma 3.3, we have
$$   
{\rm rank}(P_{\partial C^{(n)}}^{B})={\rm rank}(P_{\partial C^{(n)}}^{0}) , \quad 
{\rm rank}(P_{\Gamma^{(\pm n)}}^{B})={\rm rank}(P_{\Gamma^{(\pm n)}}^{0}),\quad {\rm for}\,\, n\geq N_0.
$$
Next, we conclude from  \eqref{compzonespectre} and \eqref{realzonespectre}  
that ${\mathcal C}\cup {\mathcal I}\subset C^{(N_0)}\bigcup\left(\underset{{|k|\geq N_0} }{\cup}{\rm Int}(\Gamma^{(k)})\right)$, 
hence that $\sigma({\mathcal A}_B)$ is a subset of  $C^{(N_0)}\bigcup\left(\underset{{|k|\geq N_0} }{\cup}{\rm Int}(\Gamma^{(k)})\right)$. 
Now Theorem \ref{countinglemma}  follows from the fact that
$$
 {\rm rank}(P_{\partial C^{(N_0)}}^{0})=2N_0\quad {\rm and}\quad {\rm rank}(P_{\Gamma^{(\pm n)}}^{0})=1.
$$
\hfill{$\square$}

\begin{remark} 
In the proofs of Lemmas \ref{resolventestimate}-\ref{lemR},  we have used only the fact that  the distance between 
two consecutive eigenvalues of ${\mathcal A}_0$ tends to infinity at infinity and the fact that ${\mathcal A}_0$ is a skew-adjoint operator.  
Similar general result are well-known $($see  Theorem 4.15a  in \cite{Ka66_01}$)$.
\end{remark}

\subsection[Riesz basis]{Riesz basis}\label{rootvect}

We start this subsection by constructing the eigenvectors associated to the high frequencies of ${\mathcal A}_B$.
Since the high frequencies of ${\mathcal A}_B$ are simple
then for all $k\in \mathbb N^*$, $k>N_0$ ($N_0$ given by Theorem \ref{countinglemma}),
we define 
\begin{equation}\label{higheigenvec}
\varphi_{\pm k}=P_{\pm k}^BV_{\pm k},
\end{equation} 
where $V_{\pm k}$ is the eigenvector of ${\mathcal A}_0$ associated to the eigenvalue $\pm i\sqrt{\mu_k}$
given by (\ref{unspectra}), and $P_{\Gamma^{(\pm k)}}^{B}$ is given by (\ref{projector}).  Note that $P_n^0V_n=V_n$
for all $n\in \mathbb Z^*$.

For $n\in \mathbb Z^*$, we denote the eigenvalue of ${\mathcal{A}}_B$ by $\lambda_n(B)$. 
We have the following proposition:
\begin{proposition}\label{quadratic} 
We assume \ref{A1} and that \eqref{A3} is satisfied.
For $k\in \mathbb Z^*,$ such that $|k|>N_0$, the function $\varphi_k$ is an eigenvector of ${\mathcal{A}}_B$
associated to the eigenvalue $\lambda_k:=\lambda_k(\mathcal{A}_B)$. Moreover, there exists $C>0$ such that
\begin{equation}\label{higheigenvectorest}
\Vert \varphi_n-V_n\Vert_{{\mathcal{H}}}\leq  
 C\frac{\delta_{n}}{\delta_{n-1}^2},\quad \hbox{\rm for all} \ n \quad \hbox{\rm such that} \ |n|>N_0. 
\end{equation}
Here $N_0$ is given by Theorem \ref{countinglemma}.\\
In particular, $\Vert \varphi_n\Vert_{\mathcal{H}}=1+o(1)$ uniformly for $n\in \mathbb Z^*$, $|n|>N_0$.
\end{proposition}

\begin{proof}
For all $m\in \mathbb Z^*$, $|m|>N_0$, we have ${\mathcal{A}}_B\varphi_{m}={\mathcal{A}}_BP_{\Gamma^{( m)}}^{B}V_{m}=
\lambda_{m}(B)P_{\Gamma^{(m)}}^{B}V_{m}=\lambda_{m}(B)\varphi_{m}$.
Using Lemma \ref{lemR} and the fact that $P_{\Gamma^{(n)}}^{0}V_n=V_n$ with $\Vert V_n\Vert_{V\times L^2}=1$, we get:
$$
\Vert \varphi_m-V_m\Vert_{\mathcal{H}}=\Vert  (P_{\Gamma^{(m)}}^{B}-P_{\Gamma^{(m)}}^{0})V_m\Vert_{\mathcal{H}}
\leq \Vert P_{\Gamma^{(m)}}^{B}-P_{\Gamma^{(m)}}^{0}\Vert_{{\mathcal L}({\mathcal{H}})}\leq  C\frac{\delta_{m}}{\delta_{m-1}^2},
$$
for all $m\in \mathbb Z^*$, $|m|>N_0$, ($C$ independent of $m$). 
In particular, parallelogram inequality and recalling that $\Vert V_m\Vert_{\mathcal{H}}=1$
give that
 $\Vert \varphi_m\Vert_{\mathcal{H}}=1+o(1)$ uniformly for $m\in \mathbb Z^*$, $|m|>N_0$.
\end{proof}

Now, we complete the sequence $(\varphi_k)_{|k|>N_0}$ of the eigenvectors associated to the high frequencies of ${\mathcal A}_B$
by considering the generalized eigenvectors associated to the low frequencies of ${\mathcal A}_B$.
Note that the number of these  generalized eigenvectors associated to the low frequencies of ${\mathcal A}_B$ is finite, 
at most $2N_0$ by Theorem \ref{countinglemma}. 
For $k\in \mathbb Z^*$ such that $|k|\leq N_0$, 
we denote by $m_k$ the algebraic multiplicity of $\lambda_k:=\lambda_k(\mathcal{A}_B)$ and we
associated to it the Jordan chain of generalized eigenvectors,
$\big(W_{k,p}\big)_{p=0}^{m_k-1}$, i.e., a Jordan basis of the root 
subspace ${\mathcal E}_k:=\Big\{W\in \mathcal{H};\
\big({\mathcal A}_B-\lambda_k\big)^{m_k}W=0\Big\}$,
\begin{eqnarray}\label{eigfunc1}
& {\mathcal A}_BW_{k,l}=\lambda_k W_{k,l},\quad \Big\langle
W_{k,l},W_{k,{l'}}\Big\rangle=0,\quad l'<l=0,\cdots,p_k.       
\\
& {\mathcal A}_BW_{k,m}=\lambda_k W_{k,m}+W_{k,m-1},\quad \Big\langle
W_{k,m},W_{k,{m'}}\Big\rangle=0,\, 
\\
& 0\leq m'<m=p_k+1,\cdots,m_k-1.\nonumber
\end{eqnarray}
Here $p_k$ is the dimension of the eigenspace ${E}_k:=\Big\{W\in \mathcal{H};\
\big({\mathcal A}_B-\lambda_k\big)W=0\Big\}$, ${E}_k\subset {\mathcal E}_k$.

Now, we take the family of generalized  eigenvectors of ${\mathcal A}_B$:
$$
{\mathbb B}:=\big(W_{k,p}\big)_{|k|\leq N_0, 0\leq p\leq m_k-1}\cup
\big(\varphi_n\big)_{|n|>N_0}. 
$$

Since ${\overline{{\rm Vect}({\mathbb B})}}=\mathcal{H}$ 
(see Proposition \ref{propspectrales}, (iii)) and by assumption \ref{A3} the family ${\mathbb B}$ is quadratically close to the orthonormal
basis $\big(V_k\big)_{k\in \mathbb Z^*}$ of eigenvectors of the operator $\mathcal{A}_0$
(see \eqref{higheigenvectorest}). Then it follows from the Fredholm Alternative, see e.g., \cite[Appendix D, Theorem 3]{PoTr86_01}, 
the following result:
\begin{theorem}\label{Rieszbasis}
Assume \ref{A1} and \ref{A2}.
Then the set $\displaystyle {\mathbb B}$ is a Riesz basis for the energy
space $\mathcal{H}$. Moreover, there exists a linear isomorphism
$\Phi$ of $\mathcal{H}$ such that for all $n\in \mathbb Z^*$, $|n|>N_0$, $\Phi V_n=\varphi_n$ 
and $\Phi\Big({\rm Vect}(V_n,\, |n|\leq N_0)\Big)={\rm Vect}(W_{k,p},\, |k|\leq N_0, 0\leq p\leq m_k-1)$.
\end{theorem}

\subsection{End of the proof of Main result}\label{PR}

Using Theorem \ref{Rieszbasis}, we may expand the initial data as
$$
\big[u^0,v^0\big]=\sum_{|k|\leq N_0}\sum_{p=0}^{m_k-1}c_{k,p}
 W_{k,p}+\sum_{|n|> N_0}c_{n}
\varphi_{n}\,.
$$
Then the solution of \eqref{Eqd3} is given by
\begin{equation}
\big[u,\partial_t u\big]= \sum_{|k|\leq N_0}\exp(\lambda_k t)
\sum_{p=0}^{m_k-1} c_{k,p} \sum_{l=0}^p
\frac{t^{p-l}}{(p-l)!} W_{k,l}+\sum_{|n|> N_0}c_{n}\exp(\lambda_n t)\varphi_{n}\,.
\end{equation}

Recalling from Theorem \ref{countinglemma} that at most $2N_0$ eigenvalues may
be of algebraic multiplicity greater than one and that $2N_0$ is the
maximum of such multiplicity, and the family $\big(V_{\pm k}\big)_{k\in \mathbb N^*}$ 
is an orthonormal basis of the energy space $\mathcal{H}$ (see Lemma \ref{undampedspectra}),
 then, by the linear isomorphism $\Phi$, we get  
$$
E\big(u(t)\big)=\Big\Vert\big[u,\partial_t u\big]\Big\Vert_{\mathcal{H}}^2 \leq \Vert \Phi\Vert^2\Vert \Phi^{-1}\Vert^2
\big(1+t^{2N_0}\big)\exp\big(2\mu(B)t\big)E\big(u(0)\big).
$$
Then $\omega(B)\leq \mu(\mathcal{A}_B),$ this with inequality \eqref{IN1} we
have established our main result.\hfill$\square$

\begin{remark} Note that, we talk about under (resp. over) damping if 
$\frac{1}{2}\Vert B^*\Vert^2_{{\mathcal L}(H,U)}$ is less (resp. greater) than $\sqrt{\mu_1}$, 
see \cite{CoZu94_01}. Recall that $\mu_1>0,$ is the first eigenvalue of $A$.
\end{remark}

\section{Some applications}
Firstly, we give examples of dissipative systems which satisfy Assumption \ref{A3} and we deduce the main result for these samples.
In the second part, we extend our result to some non-dissipative systems and we give an example that illustrates this situation.
  
\subsection{Damped Euler-Bernoulli beam equation} \label{example1}

We consider the following system:
\begin{equation}\label{eq1}
\partial^2_t u (x,t) + \partial^4_x u(x,t) +
2a(x)\partial_t u(x,t)= 0,\quad
0 < x < 1, \ t > 0,
\end{equation}
\begin{equation}\label{eq2}
u(0,t) = u(1,t) = 0, \quad \partial^2_x u(0,t) = \partial^2_x u(1,t) = 0, \quad t > 0,
\end{equation}
\begin{equation}\label{eq3}
u(x,0) = u^0(x), \quad \partial_t u(x,0) = u^1(x), \quad
0 < x < 1,
\end{equation}
where $a \in L^\infty(0,1)$ is non-negative satisfying the following condition: 
\begin{equation}\label{condexp}
\exists \, c>0 \hbox { s.t., } a(x) \geq c,\,\,  \; \hbox{a.e.,\, in  a non-empty open subset}  \; I \, \hbox{of}
\;  (0,1).
\end{equation}

We define the energy of a solution $u$ of 
\eqref{eq1}-\eqref{eq3}, at time  $t$,  as
\begin{equation}\label{DefEnergy}
E\big(u(t)\big)=\frac{1}{2}\int_{0}^1 \left( \big|\partial_t u (x,t) \big|^2 +
\big|\partial_x^2 u(x,t)\big|^2\right)\, dx\,.
\end{equation}

$$
U = L^2(0,1), \, H= L^2(0,1), \, H_{\frac{1}{2}} = H^2(0,1) \cap H^1_0(0,1), $$
$$
{\mathcal D}(A) = \left\{u \in H^4(0,1) \cap H^1_0(0,1); \frac{d^2u}{dx^2} (0) = \frac{d^2u}{dx^2} (1) = 0 \right\}, 
$$
$$
{\mathcal H} = [H^2(0,1) \cap H^1_0(0,1)] \times L^2(0,1), 
$$
$$
A = \frac{d^4}{dx^4}, \quad B \phi = B^* \phi = \sqrt{2a(x)}\phi, \quad \forall \phi \in L^2(0,1).
$$
So,
$$
{\mathcal A}_0 = \left(
\begin{array}{cc}
0 & I \\
- \frac{d^4}{dx^4} & 0
\end{array}
\right), \; {\mathcal A}_{B} = \left(
\begin{array}{cc}
0 & I \\
- \frac{d^4}{dx^4} & - 2a(x)
\end{array}
\right).
$$

\begin{itemize}
 
\item The operator ${\mathcal A}_0$ is skew-adjoint and with compact inverse and
the spectrum is given by $\sigma({{\mathcal A}_0}) = \left\{\pm i k^2 \pi^2, k \in \mathbb{N}^* \right\},$
then Assumptions \ref{A1} and \ref{A2} are satisfied. 

\item Note that the inequality (\ref{Esti}) is satisfied according to
\cite{Ha89_01}, if $a$ satisfies \eqref{condexp}.
So, $\omega(B) < 0$. 

\item As a direct implication of Theorem \ref{Princ}, we have the following result 
(this result was proved in \cite{AmDiZe13_01}):
\end{itemize}

\begin{proposition} The fastest decay rate is given by the spectral abscissa, i.e.,
$$
\omega(B) = \mu({\mathcal A}_{B}).
$$
\end{proposition}

\subsection{Extension to non-dissipative systems}

\vskip1pt
We consider the system described by:
\begin{equation}\label{Eqd1s} 
\ddot{x}(t) + A x(t) + Kx(t) = 0,\quad
\big(x(0),\dot{x}(0)\big)=(x_0,x_1)\in H_{\frac{1}{2}}\times H,
\end{equation}
where $A$ is the same operator as above and $K \in {\mathcal L}(H_{\frac{1}{2}},H).$

We can rewrite the system (\ref{Eqd1s}) as a first order differential equation, by putting 
$Y(t)={}^T\big(x(t),\dot{x}(t)\big)$:
\begin{equation}
\label{Eqd3s} 
\dot{Y}(t) + {\mathcal A}_{K} Y(t)=0, \quad  Y(0)={}^T(x_0,x_1)\in {\mathcal H},
\end{equation}
where ${\mathcal A}_{K}:={\mathcal A}_0-{\mathcal K} : {\mathcal D}({\mathcal A}_{{\mathcal K}}) 
= {\mathcal D}({\mathcal A}_0)\subset {\mathcal H} \rightarrow {\mathcal H},$ with
$$\displaystyle
{\mathcal A}_0= \left(
\begin{array}{cc}
0  & I \\
- A & 0
\end{array}
\right) : {\mathcal D}({\mathcal A}_0) = {\mathcal D}(A) \times H_{\frac{1}{2}} \subset {\mathcal H}  \rightarrow {\mathcal H},$$ and
$\mathcal{K} = \left( \begin{array}{cc} 0  & 0 \\ - K & 0
 \end{array} \right) \in {\mathcal L}({\mathcal H}).$
  
The system (\ref{Eqd1s}) is well-posed. More precisely, the following classical result holds.
\begin{proposition}\label{exists}
Suppose that $(x_0,x_1) \in {\mathcal H}$. Then the
problem (\ref{Eqd1}) admits a unique solution $x$ in the following space $C\big([0,+\infty);H_{\frac{1}{2}} \big)\cap C^1\big([0,+\infty);H\big).$
\end{proposition}

We denote, $$E\big(x(t)\big)=\frac{1}{2}\Big\Vert\big(x(t),\dot{x}(t)\big)\Big\Vert^2_{\mathcal H}.$$

Let $\mu(K)$ be the {\it spectral abscissa} of ${\mathcal A}_K$ given by:
\begin{equation}\label{abscissespecs}
\mu(K)= \sup \big\{{\rm Re}(\lambda);\ \lambda \in \sigma({\mathcal A}_K) \big\}.
\end{equation}
Here $\sigma({\mathcal A}_K)$
denotes the spectrum of ${\mathcal A}_K$. 

We define the growth bound, depending on $K$, as 
$$
\omega({K})=\inf\big\{\omega;\ \hbox{there exists}\,\, C=C(\omega)>0\, \hbox{such that}$$
\begin{equation}\label{DefRates}
E(x(t))\leq C(\omega) \, e^{2\omega t}E(x(0))
\hbox{ for every solution of (\ref{Eqd1s}) with initial data in}\ {\mathcal H}\big\}.
\end{equation}
As above we can prove the following result.
\begin{theorem} \label{princb}
Assume \ref{A1} and \ref{A2}. Then,
\begin{itemize}
\item[(i)] The eigenvectors of the
associated operator ${\mathcal A}_K$ form a Riesz basis in the energy space ${\mathcal H}$
\item[(ii)]
\begin{equation}\label{princss}
\omega(K)=\mu(\mathcal{A}_K).
\end{equation} 
\end{itemize}
\end{theorem}

\noindent
{\bf Example~:} {\it Euler-Bernoulli equation with force term}

We consider the following initial and boundary value problem: 
\begin{equation}
 \label{eq1a} 
\partial^2_t u (x,t) + \, \partial^4_x  u (x,t) +  p \, \partial_x^2 u (x,t) = 0, \,  
0 < x < 1, \, t > 0, 
\end{equation}
\begin{equation}\label{eq2a} 
u (0,t) = \partial_x u (0,t) = 0, \, \partial^2_x u(1,t) = 0, \, \partial_x^3 u(1,t) = 0, \, t > 0,  
\end{equation}
\begin{equation}\label{eq3a} 
u (x,0) = u^0(x,0), \, \partial_t u(0,x) = u^1(x), \,  
0 < x < 1, 
\end{equation}
where $p$ is a positive constant.  

Here,
$$
H= L^2(0,1), \, H_{\frac{1}{2}} = \left\{u \in H^2(0,1); u(0) =0, \, \frac{du}{dx}(0) = 0 \right\},
$$ and 
the operators are  defined        
$$
\mathcal{A}_0 = \left(
\begin{array}{cc}
0 &  Id  \\
 -\frac{d^4}{dx^4}  - p \, \frac{d^2}{dx^2} & 0 
\end{array}
\right),
$$
$$
{\mathcal D}(\mathcal {A}_0) = \left\{ (u,v) \in \left(H^4(0,1)\cap H_{\frac12}\right)\times  H_{\frac12},\,\,  \frac{d^2u}{dx^2}(1) = 0, \, \frac{d^3u}{dx^3}(1) = 0\right\},
$$
and $K = p \, \frac{d^2}{dx^2} \in {\mathcal L}(H_{\frac{1}{2}},H)$.
 
We have the for all $(u^0,u^1)\in  H_{\frac{1}{2}} \times L^2(0,1)$ the  
problem \eqref{eq1a}-\eqref{eq3a} 
admits a unique solution  
$$u \in C([0,+\infty);H_{\frac{1}{2}})\cap C^1([0,+\infty);L^2(0,1)).$$

The spectrum of $\mathcal{A}_0$ is given by $(\pm i k^2\pi^2)_{k\in \mathbb{N}^*}$. Then
Assumptions \ref{A1} and \ref{A2} are satisfied. Therefore, according to Theorem \ref{princb}, we obtain

\begin{proposition}

\begin{itemize}
\item[{}] {}
\item[(i)] The generalized eigenvectors of the associated operator ${\mathcal A}_K$ form a Riesz basis in the energy space ${\mathcal H}.$
\item[(ii)] $\omega(K)=\mu(\mathcal{A}_K).$
\end{itemize}

\end{proposition}

\end{document}

%% file: fig1.pstex_t
\begin{picture}(0,0)%
\includegraphics{fig1.eps}%
\end{picture}%
\setlength{\unitlength}{1036sp}%
\begingroup\makeatletter\ifx\SetFigFont\undefined%
\gdef\SetFigFont#1#2#3#4#5{%
  \reset@font\fontsize{#1}{#2pt}%
  \fontfamily{#3}\fontseries{#4}\fontshape{#5}%
  \selectfont}%
\fi\endgroup%
\begin{picture}(12238,15546)(272,-7868)
\put(6436,884){\makebox(0,0)[lb]{\smash{{\SetFigFont{5}{6.0}{\rmdefault}{\mddefault}{\updefault}{\color[rgb]{0,0,0}$i\sqrt{\mu_n}$}%
}}}}
\put(6616,-3796){\makebox(0,0)[lb]{\smash{{\SetFigFont{5}{6.0}{\rmdefault}{\mddefault}{\updefault}{\color[rgb]{0,0,0}$i\sqrt{\mu_{n-1}}$}%
}}}}
\put(6481,5294){\makebox(0,0)[lb]{\smash{{\SetFigFont{5}{6.0}{\rmdefault}{\mddefault}{\updefault}{\color[rgb]{0,0,0}$i\sqrt{\mu_{n+1}}$}%
}}}}
\put(8686,-6400){\makebox(0,0)[lb]{\smash{{\SetFigFont{5}{6.0}{\rmdefault}{\mddefault}{\updefault}{\color[rgb]{0,0,0}$\frac{\delta_n}{2}$}%
}}}}
\put(3781,-6400){\makebox(0,0)[lb]{\smash{{\SetFigFont{5}{6.0}{\rmdefault}{\mddefault}{\updefault}{\color[rgb]{0,0,0}$-\frac{\delta_n}{2}$}%
}}}}
\put(6436,-1794){\makebox(0,0)[lb]{\smash{{\SetFigFont{5}{6.0}{\rmdefault}{\mddefault}{\updefault}{\color[rgb]{0,0,0}$ia_n$}%
}}}}
\put(8821,884){\makebox(0,0)[lb]{\smash{{\SetFigFont{5}{6.0}{\rmdefault}{\mddefault}{\updefault}{\color[rgb]{0,0,0} $b_n$}%
}}}}
\put(6436,3494){\makebox(0,0)[lb]{\smash{{\SetFigFont{5}{6.0}{\rmdefault}{\mddefault}{\updefault}{\color[rgb]{0,0,0}$ia_{n+1}$}%
}}}}
\put(3300,884){\makebox(0,0)[lb]{\smash{{\SetFigFont{5}{6.0}{\rmdefault}{\mddefault}{\updefault}{\color[rgb]{0,0,0}$d_n$}%
}}}}
\put(100,-7510){\makebox(0,0)[lb]{\smash{{\SetFigFont{5}{6.0}{\rmdefault}{\mddefault}{\updefault}{\color[rgb]{0,0,0}${\bf Re}(z)=-\Vert B^*\Vert^2_{{\mathcal L}({U},{H})}$}%
}}}}
\put(6526,-6226){\makebox(0,0)[lb]{\smash{{\SetFigFont{5}{6.0}{\rmdefault}{\mddefault}{\updefault}{\color[rgb]{0,0,0}$0$}%
}}}}
\put(9946,3179){\makebox(0,0)[lb]{\smash{{\SetFigFont{5}{6.0}{\rmdefault}{\mddefault}{\updefault}{\color[rgb]{0,0,0}{\large$\Gamma^{(n)}$}}%
}}}}
\end{picture}%

%% file: Amm-Dim-Zer30-06-2014.bbl
\begin{thebibliography}{99}

\bibitem{AmHeTu00_01}
K.~Ammari, A.~Henrot, and M.~Tucsnak, \emph{Optimal location of the actuator for the pointwise stabilization of a string}, 
C. R. Acad. Sci. Paris S\'er. I Math. \textbf{330} (2000), no.~4, 275--280.

\bibitem{AmHeTu01_01}
K.~Ammari, A.~Henrot and M.~Tucsnak, \emph{Asymptotic behaviour of the solutions and optimal location of
  the actuator for the pointwise stabilization of a string}, Asymptot. Anal.
  \textbf{28} (2001), no.~3-4, 215--240.

\bibitem{AmTu01_01}
K.~Ammari and M.~Tucsnak, \emph{Stabilization of second order evolution
  equations by a class of unbounded feedback}, ESAIM Control Optim. Calc. Var.
  \textbf{6} (2001), 361--386.

\bibitem{AmDiZe13_01}
K.~Ammari, M.~Dimassi and M.~Zerzeri, \emph{The rate at which energy decays in a viscously damped 
hinged Euler-Bernoulli beam}, 
{\it accepted} for publication in Journal of Differential Equations, Juin 2014.

\bibitem{BeRa00_01}
A.~Benaddi and B.~Rao, \emph{Energy decay rate of wave equations with indefinite damping},
J. Differential Equations \textbf{161} (2000), no.~2, 337--357.   

\bibitem{CaCo01_01}
C.~Castro and S.~Cox, \emph{Achieving arbitrarily large decay in the damped
  wave equation}, SIAM J. Control Optim. \textbf{39} (2001), no.~6, 1748--1755.

\bibitem{CoZu94_01}
S.~Cox and E.~Zuazua, \emph{The rate at which energy decays in a damped
  string.}, Comm. Partial Differential Equations \textbf{19} (1994), no.~1-2,
  213--243.

\bibitem{CoZu95_01}
\bysame, \emph{The rate at which energy decays in a string damped at one end},
  Indiana Univ. Math. J. \textbf{44} (1995), no.~2, 545--573.

\bibitem{Fr99_01}
P.~Freitas, \emph{Optimizing the rate of decay of solutions of the wave
  equation using genetic algorithms: a counterexample to the constant damping
  conjecture}, SIAM J. Control Optim. \textbf{37} (1999), no.~2, 376--387.

\bibitem{GoKr69_01}
I.C. Gohberg and M.G. Kre\u{\i}n, \emph{Introduction to the theory of linear
  nonselfadjoint operators}, American Mathematical Society,, vol.~18,
  Providence, R.I., 1969.

\bibitem{Ha89_01}
A.~Haraux, \emph{Une remarque sur la stabilisation de certains syst\`emes du
  deuxi\`eme ordre en temps}, Portugal. Math. \textbf{46} (1989), no.~3,
  245--258.

\bibitem{Ka66_01}
T.~Kato, \emph{Perturbation theory for linear operators}, Spring{\-}er-Ver{\-}lag, 1966.

\bibitem{Le96_01}
G.~Lebeau, \emph{\'{E}quation des ondes amorties. algebraic and geometric
  methods in mathematical physics (kaciveli, 1993)}, Math. Phys. Stud.
  \textbf{19} (1996), 73--109.

\bibitem{PoTr86_01}
J.~P\"{o}schel and E.~Trubowitz, \emph{Inverse spectral theory}, Pure and
  Applied Mathematics, vol. 130, Academic Press, Inc., Boston, MA, 1987.

\bibitem{Ra81_01}
A-G. Ramm, \emph{On the basis property for root vectors of some nonselfadjoint operators}, 
Journal of Mathematical Analysis and Applications, \textbf{80} (1981), 57--66.

\end{thebibliography}
